\documentclass{amsart}
\usepackage{amsmath,amsthm}
\usepackage{amsfonts,amssymb}

\usepackage{enumerate}

 \usepackage{graphicx}
 \usepackage{ifpdf}
\usepackage{pdfsync}
\newtheorem{thm}{Theorem}[section]

\newtheorem{lem}[thm]{Lemma}
\newtheorem{prop}[thm]{Proposition}

\newtheorem{defn}[thm]{Definition}

\theoremstyle{remark}
\newtheorem{rem}{Remark}[section]

\makeatletter
\newcommand{\bddots}{%
  \mathinner{\mkern1mu\raise\p@\vbox{\kern7\p@\hbox{.}}\mkern2mu
    \raise4\p@\hbox{.}\mkern2mu\raise7\p@\hbox{.}\mkern1mu}}
\makeatother

 \def\tr{{\triangle}}

\def\f{\frac}
 
 \def\a{{\alpha}}
 \def\b{{\beta}}
 \def\g{{\gamma}}
 
 \def\t{{\theta}}
 \def\l{{\lambda}}

 \def\la{{\langle}}
 \def\ra{{\rangle}}

 \def\cb{{\mathbf c}}
 
 \def\mb{{\mathbf m}}

 \def\zb{{\mathbf z}}

 \def\Kb{{\mathbf K}}
 \def\Pb{{\mathbf P}}

 \def\CJ{{\mathcal J}}
 \def\CL{{\mathcal L}}

 \def\CM{{\mathcal M}}

 \def\CV{{\mathcal V}}
 
 \def\BB{{\mathbb B}}
 \def\CC{{\mathbb C}}
 
 \def\NN{{\mathbb N}}
 \def\PP{{\mathbb P}}
 \def\QQ{{\mathbb Q}}
 \def\RR{{\mathbb R}}
 
 \def\TT{{\mathbb T}}
 \def\UU{{\mathbb U}}
 
        \def\sspan{\operatorname{span}}
        
        \def\rank{\operatorname{rank}}
        \def\diag{\operatorname{diag}}

\def\tr{\mathsf{t}}

\newcommand{\wt}{\widetilde}
\newcommand{\wh}{\widehat}

\begin{document}

\title{Complex versus  real orthogonal polynomials of two variables}
\author{Yuan Xu}
\address{Department of Mathematics\\ University of Oregon\\
    Eugene, Oregon 97403-1222.}\email{yuan@math.uoregon.edu}
\thanks{The work was supported in part by NSF Grant DMS-1106113}

\date{\today}

\begin{abstract}
Orthogonal polynomials of two real variables can often be represented in complex variables. We explore the 
connection between the two types of representations and study the structural relations of complex orthogonal
polynomials. The complex Hermite orthogonal polynomials and the disk polynomials are used as illustrating 
examples. 
\end{abstract}

\maketitle

\section{Introduction} 
\setcounter{equation}{0}

For a real valued weight function $W(x,y)$ defined on a domain $\Omega \subset \RR^2$, orthogonal polynomials 
of two variables with respect to $W$ are usually defined as polynomials that are orthogonal  with respect to the inner 
product 
\begin{equation} \label{eq:ipd}
  \la f, g \ra_{W}: = \int_\Omega f(x,y) g(x,y) W(x,y)dxdy. 
\end{equation}
Fixing an order among monomials $x^ky^j$ (say, for example, the graded lexicographical order), one can apply the
Gram-Schmidt process to generate an orthogonal polynomial basis with respect to this inner product. Structures of 
orthogonal polynomials so derived are well studied (cf. \cite{DX, Suetin}). It is known, for example, that they satisfy 
three--terms relations with respect to the total degree and, conversely, the three--term relations, 
together with mild conditions imposed on their coefficients, characterize the orthogonality of these polynomials; in other 
words, an analogue of Favard's theorem in one variable holds \cite{DX}. 

Another way of studying orthogonal polynomials of two variables is to express them in one complex variable, for which
we identify $\RR^2$ with the complex plane $\CC$ by setting $z = x + i y$ and regard $\Omega$ as a subset of $\CC$. 
We then consider polynomials in $z$ and $\bar z$ that are orthogonal with respect to the inner product
\begin{equation} \label{eq:ipdC}
   \la f, g \ra_W^\CC: = \int_\Omega f(z) \overline{g(z)} w(z)dxdy, 
\end{equation}
where $w(x+iy) = W(x,y)$ is the real weigh function. For example, the complex Hermite polynomials introduced 
in \cite{Ito} are orthogonal with respect to $e^{-|z|^2}$ on $\CC$, and the well--known Zernike polynomials or 
disk polynomials (\cite{Z, Z2}) are orthogonal with respect to $(1-|z|^2)^\l dxdy$ on the unit disk. 

The goals of this paper are two--fold. The first one is to show that the two approaches are essentially the same
and the difference between them is a matter of changing bases. For problems that do not require detail knowledge 
on individual elements of an orthogonal polynomial basis, such as reproducing kernels or convergence of orthogonal 
expansions, the two approaches give exactly the same result.
On the other hand, expressing orthogonal polynomials in complex variable can be more 
convenient, even essential, in some cases, and may result in more elegant formulas and relations. The connection between 
the two approaches is simple and has been worked out in some special cases, but it does not seem to be well--known
as can be seen from the disk polynomials and complex Hermite orthogonal polynomials. There have been continuous 
interests in these two families of polynomials, as can be seen from a number of recent papers 
(\cite{CGG, Gh1,Gh, Inti,Ismail, TSH,W} and their references); their identification to the corresponding orthogonal 
polynomials of two real variables, however, is hardly mentioned. 

Our second goal is to explore the structural relations for the complex orthogonal polynomials. Such relations, 
such as three--term relations and recursive relations, take different forms when expressed in complex variable. 
One may then ask the question of how Favard's theorem or other results that depend on the three--term relations 
can be stated in complex version. 

The space of orthogonal polynomials of a fixed degree in two variables can have many distinguished bases, some are
easier to work with than others. The study of complex Hermite polynomials and disk polynomials has demonstrated 
that orthogonal basis in complex version can possess elegant relations and formulas that could reveal hidden relations 
not easily seen in bases of real variables. In some other cases, for example, orthogonal polynomials on the domain
bounded by the deltoid curve, it is much easier to study orthogonal polynomials in complex variables. 

The paper is organized as follows. In the next section, we give a short expository on orthogonal polynomials of two 
real variables and illustrate the result using real Hermite and disk polynomials. The complex orthogonal polynomials 
are studied in Section 3 and their connection to real orthogonal polynomials is explained in Section 4. The structural 
relations of complex orthogonal polynomials are explored in Section 5.  

\section{Orthogonal polynomials of two variables} 
\setcounter{equation}{0}

We explain the basics of orthogonal polynomials of two real variables in this section. Our main reference is 
\cite{DX}.

Let $W(x,y)$ be a nonnegative weight function defined on a subset $\Omega \in \RR^2$, such that 
$\int_\Omega f^2(x) W(x,y) dxdy > 0$ for all nonzero $f \in \Pi^2:=\RR[x,y]$. Throughout this paper, we normalize 
$W$ so that $\int_\Omega W(x,y)dxdy =1$. Define the inner product $\la f, g \ra_W$ as in \eqref{eq:ipd},
$$
  \la f, g\ra_W = \int_\Omega f(x,y) g(x,y) W(x,y) dxdy,
$$
and assume that it is well defined for all polynomials. Let $\Pi_n^2$ denote the space of polynomials of degree 
at most $n$ in two real variables. A polynomial $P \in \Pi_n^2$ is called orthogonal if 
$$
    \la P, Q \ra_W = 0, \quad \hbox{for all $Q \in \Pi_{n-1}^2$},
$$
that is, $P$ is orthogonal to all polynomials of lower degrees. We denote by $\CV_n^2$ the space of 
orthogonal polynomials of degree $n$, 
$$
     \CV_n^2 := \sspan\{P \in \Pi_n^2: \la P,Q \ra_W =0, \,  \forall Q \in \Pi_{n-1}^2\}. 
$$
We sometimes write $\CV_n^2(W)$ to emphasis the dependence on $W$. It follows that
$$
     \dim \CV_n^2 = \# \{x^n, x^{n-1}y, \ldots, x y^{n-1}, y^n\} = n+1.
$$ 
A basis of $\CV_n^2$ is often denoted by $\{P_{k,n}: 0 \le k \le n\}$. If, additionally, $\la P_{k,n}, P_{j,n}\ra = 0$ for $j \ne k$, 
the basis is called a mutually orthogonal basis, and if, in further addition, $\la P_{k,n},P_{k,n}\ra =1$ for $0 \le k \le n$, 
the basis is called an orthonormal basis. A convenient notation is $\PP_n$, a column vector defined by 
$$
\PP_n  = (P_{0,n}, P_{1,n}, \ldots, P_{n,n})^\tr,
$$
where $\tr$ in the superscript denotes transpose. By definition, $\PP_n \PP_m^\tr$ is a matrix of $(n+1) \times 
(m+1)$. That $\{P_{k,n}: 0 \le k \le n\}$ is a basis of $\CV_n^2$ is equivalent to $\la \PP_n, \PP_m^\tr \ra
 = 0$ for $0 \le m \le n-1$, and it is an orthonormal basis of $\CV_n^2$ if, in addition, $\la \PP_n, \PP_n^\tr \ra$ is an 
identity matrix. 

In contrast to one variable, there could be many distinct bases for the space $\CV_n^2$. In fact, if $\PP_n$ consists of 
a basis of $\CV_n^2$, then for any non-singular matrix $M \in \RR^{n+1,n+1}$, $M \PP_n$ also consists of a basis of
$\CV_n^2$. Some bases of $\CV_n^2$, however, can be given by simpler formulas and easier to use than others. We 
illustrate this point with two examples that will also be used in the next two sections. 

\medskip\noindent
{\bf Example 2.1.}
{\it Hermite polynomials}. These are orthogonal with respect to the weight function
$$
   W_H(x,y):= \frac{1}{\pi} e^{-x^2 - y^2}, \qquad (x,y) \in \RR^2.
$$
There are several well--known orthogonal bases; we give two below. The first one is given by the product 
Hermite polynomials,
\begin{equation} \label{Hermite1}
  H_k(x) H_{n-k}(y) = (-1)^n e^{x^2+y^2} \left(\frac{\partial}{\partial x}\right)^k \left(\frac{\partial}{\partial y}\right)^{n-k} e^{-x^2 -y^2},\quad 0 \le k \le n.  
\end{equation}
This is a mutually orthogonal basis of $\CV_n^2(W_H)$. The second one is given in the polar coordinates 
$(x,y) = (r\cos \t, r \sin \t)$ with $0 \le \theta \le 2 \pi$ and $r \ge 0$, 
\begin{align} \label{Hermite2}
\begin{split}
  H_{j,n-2j}^{(1)} (x,y):= L_j^{(n-2j)}(r^2) r^{n -2j} \cos (n-2 j) \theta, \quad 0 \le j \le n/2, \\
    H_{j,n-2j}^{(2)} (x,y):=L_j^{(n-2j)}(r^2) r^{n-2j} \sin (n-2 j) \theta, \quad 0 \le j < n/2, 
\end{split}
\end{align}
where $L_j^{(\a)}$ denotes the usual Laguerre polynomial with parameter $\a$. That these are indeed polynomials of 
degree $n$ in $x$ and $y$ can be seen by $r = \sqrt{x^2+y^2}$ and writing $r^k \cos k \t = r^kT_k(\frac{x}{r})$ and 
$r^k \sin k \t = y r^{k-1} U_{k-1}(\f{x}{r})$, where $T_k$ and $U_k$ are the Chebyeshev polynomials of the first and the 
second kind, respectively. The polynomials in \eqref{Hermite2} consist of a mutually orthogonal basis of $\CV_n^2(W_H)$.  
\qed

\medskip\noindent
{\bf Example 2.2.}
{\it Orthognal polynomials}. These are orthogonal with respect to the weight function
$$
   W_\mu(x,y):= \frac{\l+1}{\pi} (1-x^2 - y^2)^\l, \qquad \l > -1, 
$$
for $(x,y) \in \BB^2:=\{(x,y): x^2+y^2 \le 1\}$. There are many well--known bases for this weight function. We give 
two bases that are in the same spirit as those in Example 2.1. The first one is given by
\begin{equation} \label{Disk1}
      (1-x^2-y^2)^{-\l} \frac{\partial^n }{\partial x^k \partial y^{n-k}} \left[
             (1-x^2-y^2)^{n+\l}\right], \quad 0\le k\le n. 
\end{equation}
This is a basis of $\CV_n^2(W_\l)$ but it is not a mutually orthogonal one. The second basis is given in polar 
coordinates $(x,y) = (r\cos \t, r \sin \t)$ with $0 \le \theta \le 2 \pi$ and $r \ge 0$, 
\begin{align} \label{Disk2}
\begin{split} 
       P_{j}^{(\l,n-2j)}(2r^2 -1)r^{n-2j}\cos (n-2j)\t, \quad 0 \le j \le n/2,  \\
       P_{j}^{(\l,n-2j)}(2r^2 -1)r^{n-2j}\sin (n-2j) \t,  \quad 0 \le j < n/2. 
\end{split}
\end{align}
The polynomials in \eqref{Disk2} consist of a mutually orthogonal basis of $\CV_n^2(W_\l)$. \qed
\medskip

Our definition of orthogonality can be extended to a positive definite linear functional $\CL$ defined on $\Pi^d$, which
satisfies $\CL(p^2) > 0$ whenever $p \in \Pi^d$ and $p \ne 0$. Given such a linear functional, we can define an inner
product $\la f, g \ra = \CL(f g)$, which allows us to consider orthogonal polynomials with respect to $\CL$. If the linear
functional is given by $\CL(f) = \int_\Omega f(x,y) W(x,y) dxdy$, we are back to orthogonal with respect to $W$. 

There is an analog of three--term relations for orthogonal polynomials in two variables, given in terms of $\PP_n$,
which has the simplest form for orthonormal polynomials, and it in fact characterizes the orthogonality in the sense
of Favard's theorem. Let $\CM(n,m)$ denote the set of real matrices of size $n\times m$.  

\begin{thm} \label{thm:nFavard}
Let $\{\PP_n\}_{n=0}^\infty = \{P_{k,n}: 0 \le k \le n, n\in \NN_0\}$, $\PP_0 =1$, be an arbitrary sequence in $\Pi^2$. Then the following statements are equivalent.
\item{}\hspace{.1in} (1).  There exists a positive definite linear functional $\CL$ on $\Pi^2$ which 
makes $\{\PP_n\}_{n=0}^\infty$ an orthonormal basis in $\Pi^d$.
\item{}\hspace{.1in} (2). For $n \ge 0$, $1\le i\le d$, there exist matrices $A_{n,i} \in \CM(n+1, n+2)$ and
$B_{n,i}\in \CM(n+1,n+1)$ such that
\begin{align} \label{three-term}
 \begin{split}
   x \PP_n(x,y) & = A_{n,1} \PP_{n+1}(x,y) + B_{n,1} \PP_n(x,y) + A_{n-1,1}^\tr \PP_{n-1}(x,y), \\
   y \PP_n(x,y) & = A_{n,2} \PP_{n+1}(x,y) + B_{n,2} \PP_n(x,y) + A_{n-1,2}^\tr \PP_{n-1}(x,y).
\end{split}
\end{align}
and the matrices in the relation satisfy the rank condition 
$$
   \rank A_{n,1} = \rank A_{n,2} = n+1 \quad \hbox{and} \quad \rank \left [ \begin{matrix} A_{n,1} \\ A_{n,2} \end{matrix} \right]=n+2.
$$
\end{thm}

If $\{P_{k,n}: 0 \le k \le n\}$ is a basis of $\CV_n^2$, then the matrix $H_n:=\la \PP_n, \PP_n^\tr\ra$ is positive
definite. It follows that $\wt \PP_n = H_n^{-1/2} \PP_n$ consists of an orthonormal basis of $\CV_n^d$. If $H_n$ is 
an identity matrix, then $\PP_n$ consists of an orthonormal basis of $\CV_n^2$. The three-term relations and Favard's 
theorem can be stated for non-orthonormal bases, for which $A_{n-1,i}^t$ in \eqref{three-term} needs to be replaced by
$C_{n,i} : = H_n A_{n-1,i}^\tr H_{n-1}^{-\tr}$. 

The weight function $W$ is called centrally symmetric if $W(x) = W(-x)$ and $-x \in \Omega$ whenever $x \in \Omega$. 
If $W$ is centrally symmetric, then it is known that  $B_{n,i} = 0$ in the three--term relations \eqref{three-term}.

Given an orthonormal basis $\{P_{k,n}: 0 \le k \le n, n =0,1,2,\ldots\}$, the reproducing kernels $\Pb_n(\cdot,\cdot)$ and 
$\Kb_n(\cdot,\cdot)$ of $\CV_n^2(W)$ and $\Pi_n^2$, respectively, in $L^2(W)$ are defined by 
\begin{equation}\label{kernelP}
  \Pb_m((x,y), (u,v)) :=  \PP_m^\tr (x,y)\PP_m(u,v) = \sum_{k=0}^m P_{k,m}(x,y) P_{k,m} (u,v) 
\end{equation}
and 
\begin{equation}\label{reprod-kernel}
      \Kb_n((x,y),(u,v)): = \sum_{m=0}^n \Pb_m((x,y), (u,v)).
\end{equation}
The kernel $\Kb_n(\cdot,\cdot)$ plays an essential role in the study of Fourier orthogonal expansions. 
It satisfies an analog of the Christoffel--Darboux formula: with $x = (x_1,x_2)$ and $y = (y_1,y_2)$, 
\begin{equation} \label{CDformula} 
   \Kb_n(x,y) = \frac{\bigl[ A_{n,i} \PP_{n+1}(x)\bigr]^\tr \, \PP_{n} (y) -  
  \PP_{n}^\tr(x) \bigl[ A_{n,i} \PP_{n+1} (y)\bigr]} {x_i - y_i}, \quad i=1,2. 
\end{equation}
Notice that the right hand side depends on $i$ where the left hand side does not. 

\section{Orthogonal polynomials in complex variables} 
\setcounter{equation}{0}

Let $W(x,y)$ be defined as in the previous section. For $z \in \CC$ we write $z = x + i y$ and consider 
$\Omega$ a subset of $\CC$. Define the weight function $w(z)$ by 
$$
       w(z) = w(x+iy): = W(x,y), \quad z \in \Omega,  
$$
which is a real function. Let $m_{k,j}$ denote the moment of $w(z)$ defined by 
\begin{equation}\label{moments}
   m_{k,j} = \int_\Omega z^k \bar{z}^j w(z) dz, \qquad j,k = 0, 1, 2, \ldots.
\end{equation}
It follows directly from the definition that $m_{k,j} = \overline{m_{j,k}}$. This shows that we need to treat 
$z$ and $\bar z$ separately, so that our polynomials are really functions in $z$ and $\bar z$. Thus, the 
complex inner product defined in \eqref{eq:ipdC} should be written as  
$$
  \la f, g \ra_W^\CC  :=  \int_\Omega f(z,\bar z) \overline {g(z,\bar z)} w(z) dxdy
$$
for polynomials of two variables in $z$ and $\bar z$. Let $\Pi^2(\CC):=\{P(z,\bar z): P \in \Pi^2\}$ and,
for $n \in \NN_0$, let $\Pi_n^2(\CC) := \{P(z,\bar z): P \in \Pi_n^2\}$. 

With respect to this inner product, a polynomial $P \in \Pi_n^2(\CC)$ is called orthogonal of degree $n$ if 
$\la P, Q\ra_W^\CC = 0$, for all $Q \in \Pi_{n-1}^2$. We denote by $\CV_n^2(W,\CC)$ the space of orthogonal 
polynomials of degree $n$ with respect to $\la \cdot,\cdot \ra_W^\CC$,  
$$
   \CV_n^2(\CC) := \{P \in \Pi_n^2(\CC): \la P, Q \ra_W^\CC =0, \, \forall Q \in \Pi_{n-1}^2(\CC)\}.
$$
We sometimes write $\CV_n^2(W,\CC)$ to emphasis the dependence on $W$. It follows that $\dim \CV_n^2(\CC) = n +1$. 

To distinguish between $\CV_n^2$ and $\CV_n^2(\CC)$, we reserve $\{P_{k,n}: 0 \le k \le n\}$ for a basis of $\CV_n^2$ 
and $\{Q_{k,n}: 0 \le k \le n\}$ for a basis of $\CV_n^2(\CC)$. If $\la Q_{k,n}, Q_{j,n}\ra = 0$ for $j \ne k$, the basis is called 
a mutually orthogonal basis, and if, in addition, $\la Q_{k,n},Q_{j,n}\ra =1$ for $0 \le k \le n$, the basis is called orthonormal.
We shall also use the notation $\QQ_n$ defined by 
$$
\QQ_n := (Q_{0,n}, Q_{1,n}, \ldots, Q_{n,n})^\tr.
$$
As in the case of $\PP_n$, that $\{Q_{k,n}: 0 \le k \le n\}$ is an orthogonal basis of $\CV_n^2(\CC)$ is equivalent to 
$\la \QQ_n, \QQ_m^\tr\ra_W^\CC = 0$ for $0 \le m \le n-1$ and $\QQ_n$ consists of an orthonormal basis if 
$\la \QQ_n, \QQ_n^\tr \ra_W^\CC$ is an identity matrix in addition. 

As it is in the case of real orthogonal polynomials, if $\QQ_n$ consists of a basis of $\CV_n^2(\CC)$, then so does 
$M \QQ_n$ for any nonsingular $(n+1)\times (n+1)$ matrix $M$. 

Given an orthonormal basis $\{Q_{k,n}: 0 \le k \le n, n =0,1,2,\ldots\}$, 
the reproducing kernels $\Pb_n^\CC(\cdot,\cdot)$ and $\Kb_n^\CC(\cdot,\cdot)$ of  $\CV_n^2(W,\CC)$ and
$\Pi_n^2(\CC)$, respectively, in $L^2(W)$ are defined by
\begin{equation}\label{kernelPC}
  \Pb_m^C(z, \zeta) := \QQ_m^\tr (z, \bar z)\overline{\QQ_m(\zeta, \bar \zeta)} 
      = \sum_{k=0}^m Q_{k,m}(z,\bar z) \overline{Q_{k,m} (\zeta, \bar \zeta)},
\end{equation}
and 
\begin{equation}\label{reprod-kernelC}
   \Kb_n^\CC(z, \zeta): = \sum_{m=0}^n \Pb_n^\CC(z, \zeta).
\end{equation}

Complex orthogonal polynomials can be directly constructed from the moments $m_{j,k}$, just like their counterpart 
in real variables \cite[Section 3.2]{DX}. For convenience, we define a column vector 
$$
    \zb^n = (z^n, z^{n-1} \bar z, \ldots, {\bar z}^{n-1}z, {\bar z}^n )^\tr, 
$$
and for $k, j \in \NN_0$, define the matrix $\mb_{\{k\},\{j\}}$ of size $(k+1) \times (j+1)$ by 
$$
    \mb_{\{k\},\{j\}} :=  = \la \zb^k, (\zb^j)^\tr\ra_{W}^\CC = \int_\Omega \zb^k  (\zb^j)^* w(z) dx dy.
$$
For $n \in \NN_0$, define the moment matrix of size $N \times N$ with $N= \binom{n+2}{2}$ by 
$$
   M_n = \left[ \mb_{\{k\},\{j\}}  \right]_{k,j=0}^n.
$$
Let $I_n$ denote the $n\times n$ identity matrix and $J_n$ denote the $n \times n$ backward identity, 
$$
     I_n = \left[ \begin{matrix} 
          1 &  & \bigcirc \\   &  \ddots &  \\ \bigcirc &  & 1 
            \end{matrix} \right] \quad {and} \quad   J_n = \left[ \begin{matrix} 
         \bigcirc  &  & 1 \\   &  \bddots &  \\   1 &  & \bigcirc 
            \end{matrix} \right].  
$$

\begin{lem}
For each $n = 0,1,2, \ldots$, the matrix $M_n$ is positive definite. Furthermore, $M_n$ satisfies
\begin{equation}\label{M=JMJ}
   M_n = \left [ \begin{matrix} J_1 &  & \bigcirc \\ & \ddots & \\ \bigcirc & & J_{n+1}  \end{matrix}\right]
       \overline{M_n} \left [ \begin{matrix} J_1 &  & \bigcirc \\ & \ddots & \\ \bigcirc & & J_{n+1}  \end{matrix}\right].      
\end{equation}
\end{lem}

\begin{proof}
Let $\cb \in \CC^N$ be a row vector. We can write $\cb = (\cb_0,\cb_1, \ldots, \cb_n)$ with 
$\cb_k \in \CC^{k+1}$ as row vectors. Then 
$$
    \cb M_n \cb^* = \sum_{k=0}^n\sum_{j=0}^n \cb_j  \mb_{\{k\},\{j\}} \cb_j^* =  \int_{\Omega}
        \left |\sum_{k =0}^n \cb_k \zb_k \right|^2 w(z) dxdy \ge 0, 
$$
and equality holds only if $\cb = 0$, since $W$ satisfies $\int_\Omega p^2(x) W(x,y) dxdy > 0$ for all nonzero $p \in \Pi_n^2$.
Hence, $M_n$ is positive definite. 

It follows directly form the definition that $\zb^n = J_{n+1} \overline{\zb^n}$, which implies that 
$\mb_{\{k\},\{j\}} = J_{k+1} \overline{\mb_{\{k\},\{j\}}} J_{j+1}$ and, consequently, the identity \eqref{M=JMJ}.
\end{proof}

For $0 \le k \le n$ and $j \in \NN$, we define the column vector $  \mb_{k,\{j\}}^n $ in $\CC^{j+1}$ by
$$
   \mb_{\{j\},k}^n := \int_\Omega \zb^j   z^{n-k} {\bar z}^k w(z) dxdy,
$$ 
and use it to define,  for each $k$, a matrix $M_{k,n}(z,\bar z)$ by 
$$
   M_{k,n}(z, \bar z) := \left[ \begin{array}{c|c}
  M_{n-1} & \begin{matrix}  {\mb_{\{0\},k}^n} \\  { \mb_{\{1\},k} }\\ \vdots \\ { \mb_{\{n-1\},k}^n} \end{matrix}\\ 
  \hline 1, \zb^*, \hdots, (\zb^{n-1})^* & z^{n-k} {\bar z}^k
      \end{array} \right].
$$
Now, for $0 \le k \le n$, we define monic polynomials $Q_{k,n} \in \Pi_n^2(\CC)$ by
\begin{equation} \label{eq:MonoQ}
Q_{k,n} (z,\bar z): = \frac{ \det M_{k,n}(z, \bar z) }{ \det M_{n-1}} =   z^{n-k} \bar z^k + R_{k,n-1}(z,\bar z),
\end{equation}
where  $R_{k, n-1} \in \Pi_{n-1}^2(\CC)$ as expanding the determinant in nominator shows. 

\begin{prop}
For $0 \le k \le n$, $Q_{k,n}$ are orthogonal polynomials in $\CV_n^2(W,\CC)$ and satisfy
\begin{equation} \label{eq:Qconjugate}
    Q_{k,n} (z,\bar z) = \overline{Q_{n-k,n}(z,\bar z)}, \quad 0 \le k \le n. 
\end{equation}
\end{prop}

\begin{proof}
For $ 0 \le j \le p < n$, computing $\int_\Omega Q_{k,n}(z,\bar z) \overline{z^j {\bar z}^{p-j}}w(z) dxdy$ shows that the
integral applies to the last row of the determinant in the nominator of $Q_{k,n}$, which becomes 
$$
 \int_\Omega  z^{p-j} {\bar z}^j (1, \zb^*, \hdots (\zb^{n-1})^*, z^{n-k} {\bar z}^k) w(z) dxdy.
$$
It follows that the first $\binom{n+1}2$ elements of this vector is the $j$-th row of the $p$-th block rows indexed 
by $\{p\}$ of $M_{n-1}$ and the last element is the $j$-th element of $\mb_{\{p\},k}^n$. Consequently, the 
determinant of the integral of $M_{k,n}(z,\bar z) \overline{z^j {\bar z}^{p-j}}$ has two identical rows and its value is zero. 
This proves that $Q_{k,n} \in \CV_n^2(W,\CC)$. 

Since $\zb^m = J_{m+1} \overline{\zb^m}$ and $\mb_{\{j\},k}^n = J_{j+1} \overline{\mb_{\{j\},k}^n}$, it is not difficult 
to verify, using \eqref{M=JMJ}, that 
\begin{equation*}
  M_{k,n}(z,\bar z) = \left [ \begin{matrix} J_1 &  & & \bigcirc \\ & \ddots & & \\   & & J_{n} & \\ \bigcirc & & & 1 \end{matrix}\right]
       \overline{M_{n-k,n}(z,\bar z) }\left [ \begin{matrix} J_1 &  & & \bigcirc \\ & \ddots & & \\   & & J_{n} & \\ \bigcirc & & & 1 \end{matrix}\right],      
\end{equation*}
which implies immediately the identity \eqref{eq:Qconjugate}.
\end{proof}

In terms of the column vector $\QQ_n$, the identities in \eqref{eq:Qconjugate} are equivalent to 
\begin{equation} \label{eq:Qconjugate2}
       \QQ_n (z,\bar z) = J_{n+1} \overline{\QQ_{n} (z,\bar z)}.
\end{equation}

Let $H_n := \int_\Omega \QQ_n(z,\bar z) (\QQ_n(z,\bar z))^* w(z) dxdy$. Then $H_n$ is a positive definite Hermitian
matrix. In particular, $H_n$ has positive real eigenvalues and there is an unitary matrix $S_n$ such that 
$$
  H_n = S_n \Lambda_n S_n^*, \qquad \Lambda_n = \diag (\l_0,\ldots, \l_{n}), 
$$
where $\lambda_0,\ldots, \l_{n}$ are the eigenvalues of $H_n$. Since all $\l_i >  0$, we can define the square root 
of $H_n$ by $H_n^{\pm \f 12}: = S_n \Lambda^{\pm \f12} S_n^*$, where 
$\Lambda_n^{\pm \f12} = \diag (\l_0^{\pm \f12},\ldots, \l_n^{\pm \f12}\,)$.

\begin{prop}
Let $Q_{k,n}$ be defined by \eqref{eq:MonoQ}. Define $\{Q_{k,n}': 0 \le k \le n\}$ by 
$$
   \QQ_n'(z,\bar z) = H_n^{-\f12} \QQ_n(z,\bar z).
$$
Then $\{Q_{k,n}': 0 \le k \le n\}$ is an orthonormal basis of $\CV_n^2(W,\CC)$ that satisfies \eqref{eq:Qconjugate}.
\end{prop}

\begin{proof}
By definition, $\QQ_n' (\QQ_n')^* = H_n^{-\f12} \QQ_n \QQ_n^* H_n^{-\f12}$, so that the integral of $\QQ'_n (\QQ_n')^*$ 
is an identity matrix. In other words, $\QQ_n'$ consists of an orthonormal basis of $\CV_n^2(W,\CC)$. We now
prove that $\QQ'_n$ satisfies \eqref{eq:Qconjugate2}.   

Sine $Q_{k,n}$ satisfies \eqref{eq:Qconjugate}, it follows by \eqref{eq:Qconjugate2} that $H_n$ satisfies
$H_n =  J_{n+1} \overline{H_n} J_{n+1}$. By its definition, $H_n^{-\f12}$ satisfies the same relation. 
Since $J_{n+1} J_{n+1} = I_{n+1}$, it follow that 
$$
   J_{n+1} \QQ_n' = J_{n+1} H_n^{-\f12} \QQ_n = J_{n+1} H_n^{-\f12} J_{n+1} \overline{\QQ_n}
      = \overline{H_n^{-\f12}\QQ_n} = \overline{\QQ_n'},  
$$
which verifies that $\QQ_n'$ satisfy \eqref{eq:Qconjugate2} and completes the proof.
\end{proof}

\begin{rem} 
It should be pointed out that not every basis of $\CV_n^2(W, \CC)$ satisfies the relation \eqref{eq:Qconjugate}. 
Indeed, suppose $\QQ_n$ consists of a basis of $\CV_n^2(W, \CC)$ that satisfies \eqref{eq:Qconjugate}, then 
$M \QQ_n$ also consists of a basis of $\CV_n^2(W,\CC)$ for every invertible matrix $M$ of size $(n+1) \times (n+1)$ 
and we can choose an $M$ so that \eqref{eq:Qconjugate2} fails to hold for $M\QQ_n$. The relation 
\eqref{eq:Qconjugate} plays an essential role in our development in the next section.
\end{rem}

Below we give two classical examples of complex orthogonal polynomials of two variables. The first one is
the complex Hermite polynomials introduced in \cite{Ito}, which have been studied by many authors, see 
\cite{Gh1, Gh, Inti, Ismail} and the references therein. All properties list below are known, although some are 
given in somewhat different forms. 

\medskip\noindent
{\bf Example 3.1.}
{\it Complex Hermite polynomials}. A classical family of polynomials that are orthogonal with respect to 
the weight function
$$
   w_H(z):= \frac{1}{\pi} e^{-|z|^2}, \qquad z \in \CC,
$$
which satisfies $w_H(z) = W_H(x,y)$, is introduced in \cite{Ito} by 
\begin{equation} \label{HermiteC}
  H_{k,j} (z,\bar z): = (-1)^{k+j} e^{z \bar z} \left( \frac{\partial}{\partial z} \right)^k 
   \left( \frac{\partial}{\partial \bar z} \right)^j e^{- z \bar z}, \quad k,j = 0,1,\ldots, 
\end{equation}
where, with $z = x + i y$, 
$$ 
 \frac{\partial}{\partial z}  = \frac12 \left ( i \frac{\partial}{\partial x} + \frac{\partial}{\partial y}  \right)
 \quad \hbox{and} \quad  \frac{\partial}{\partial \bar z}  = \frac12 \left ( i \frac{\partial}{\partial x} - \frac{\partial}{\partial y}  \right).
$$
By induction, it is not difficult to see that $H_{k,j}$ satisfies an explicit formula  
\begin{equation} \label{HermiteC-1}
  H_{k,j} (z,\bar z) = z^k {\bar z}^j {}_2 F_0 \left(-k, -j; \f{1}{z \bar z}\right )
\end{equation}
from which it follows immediately that 
\begin{equation} \label{HermiteC-2}
    H_{k,j} (z, \bar z) = \overline {H_{j,k}(z, \bar z)}. 
\end{equation}
Working with \eqref{HermiteC-1} by rewriting the summation in ${}_2 F_0$ in reverse order, it is easy to deduce
that $H_{k,j}$ can be written as a summation in ${}_1 F_1$, which leads to 
\begin{equation} \label{HermiteC-3}
      H_{k,j} (z,\bar z) = (-1)^j j! z^{k-j} L_j^{(k-j)}(|z|^2), \qquad k \ge j, 
\end{equation}
where $L_j^{(\a)}$ is again the Laguerre polynomial. Using the property \cite[(5.1.13)]{Szego} of the Laguerre 
polynomials, it follows that $H_{k,j}$ satisfy the recursive relation
\begin{equation} \label{HermiteC-4}
   z H_{k,j} (z, \bar z) = H_{k+1,j}(z,\bar z) + j H_{k,j-1}(z, \bar z). 
\end{equation}
Using polar coordinates and the orthogonality of the Laguerre polynomials, we see that 
\begin{equation} \label{HermiteC-5}
     \int_\CC H_{k,j}(z,\bar z) \overline {H_{m,l}(z,\bar z)} w_H(z)dxdy = j! k! \delta_{k,m}\delta_{j,l}.
\end{equation}
Finally, $H_{k,j} \in \CV_{k+j}^2(w_H, \CC)$ and a mutually orthogonal basis of $\CV_n^2(w_H, \CC)$ is
given by $\{H_{n-j,j}: 0 \le j \le n\}$, which satisfies \eqref{eq:Qconjugate} by \eqref{HermiteC-2}.   \qed

\medskip

Our second example is the disk polynomials, which were first introduced  by Zernik \cite{Z, Z2} in his
work in optics (for $\mu=0$) and have been extensively studied (see, for example, references in \cite{W}). 
We follow \cite[Section 2.4.3]{DX} below.

\medskip\noindent
{\bf Example 3.2.}
{\it Zernike and Disk polynomials}. These polynomials are orthogonal with respect to the weight function
$$
  w_\l(z) = \frac{\l+1}{\pi} (1- |z|^2)^\l, \quad \l > -1, \quad z \in \CC,
$$
which satisfies $w_\l(z) = W_\l(x,y)$.  For $k, j \ge 0$, we define 
$$
   P^\l_{k,j}(z,\bar z) = \frac{(\lambda+1)_{k+j}}{(\lambda+1)_k(\lambda+1)_j} z^k \bar z^j
       {}_2F_1\Big(\begin{matrix} -k,-j \cr - \lambda - k -j\end{matrix};
   \frac{1}{z \bar z}\Big), \qquad k, j \ge 0.  
$$
which shows immediately that 
\begin{equation}\label{disk-poly}
   P^\lambda_{k,j}(z,\bar z ) = \overline{P^\lambda_{j,k}(z,\bar z)}.  
\end{equation}
They can be written in terms of the classical Jacobi polynomial $P_j^{(\a,\b)}(t)$ as
\begin{equation}\label{disk-poly2}
   P^\lambda_{k,j}(z,\bar z) = 
   \frac{j!}{(\lambda+1)_j} P_{j}^{(\lambda,k-j)}(2 |z|^2 -1) z^{k-j}, \quad k > j.
\end{equation}
These polynomials satisfy a recursive relation defined by 
\begin{equation}\label{disk-poly3}
   P^\lambda_{k,j}(z,\bar z) = \frac{(\lambda+1)_{k+j}}{(\lambda+1)_k(\lambda+1)_j} z^k \bar z^j
       {}_2F_1\Big(\begin{matrix} -k,-j \cr - \lambda - k -j\end{matrix};
   \frac{1}{z \bar z}\Big), \qquad k, j \ge 0.  
\end{equation}
Furthermore, their orthogonality is given by 
\begin{equation}\label{disk-poly4} 
  \int_{\BB^2} P^\lambda_{k,j}(z,\bar z) \overline{ P^\lambda_{m,l}(z, \bar z)} d w_\lambda(z) dxdy
   = h_{k,j}^\l \delta_{k,m} \delta_{j,l},
\end{equation}
where 
$$
  h_{k,j}^\l : = \frac{\lambda+1}{\lambda+k+j+1} \frac{k! j!}{(\lambda+1)_k(\lambda+1)_j}.
$$
The polynomial $P^\lambda_{k,j} \in \CV_{k+j}^2(w_\l, \CC)$ and a mutually orthogonal basis of $\CV_n^2(w_\l, \CC)$ is
given by $\{P^\l_{n-j,j}: 0 \le j \le n\}$, which satisfies \eqref{eq:Qconjugate} by \eqref{disk-poly}.   \qed

\medskip

Comparing these two examples with Example 2.1 and Example 2.2 in Section 2 shows a close relation between
the complex and real orthogonal polynomials. In the next section, we clarify this relation. 

\section{Complex vs real orthogonal polynomials}
\setcounter{equation}{0}

In this section we establish connections between complex orthogonal polynomials and real orthogonal polynomials
of two variables.  

\begin{defn} \label{def:PvsQ}
Given $\{Q_{k,n}: 0 \le k \le n\} \in \CV_n^2(W,\CC)$, we define 
\begin{align} \label{eq:PvsQ}
\begin{split}
  P_{k,n}(x,y) & := \frac{1}{\sqrt{2}} \left[Q_{k,n} (z, \bar z) + Q_{n-k,n} (z, \bar z) \right], \quad 0 \le k \le \f n 2, \\
  P_{k,n}(x,y) & := \frac{1}{\sqrt{2} i } \left[Q_{k,n} (z, \bar z) - Q_{k-k,n} (z, \bar z) \right], \quad \f n 2 < k \le n. 
\end{split}
\end{align}
Conversely, given $\{P_{k,n}(x,y): 0 \le k \le n\} \in \CV_n^2(W)$, we define 
\begin{align} \label{eq:QvsP}
\begin{split}
   Q_{k,n}(z, \bar z) &:=   \frac{1}{\sqrt{2}} \left[P_{k,n} (x,y) - i P_{n-k,n}(x,y) \right], \quad 0 \le k \le \f n 2 \\
   Q_{k,n}(z, \bar z) &:=   \frac{1}{\sqrt{2}} \left[P_{n-k,n} (x,y) + i P_{k,n}(x,y) \right], \quad \f n 2 < k \le n,\\
   Q_{\f n 2 ,n}(z, \bar z) &:=   \frac{1}{\sqrt{2}} P_{\f n 2, n}(x,y), \quad \hbox{$n$ is even}. 
\end{split}
\end{align}
\end{defn} 

We define a matrix $L_n$ of $(n+1) \times (n+1)$ as follows: 
$$
   L_{2 m-1} :=  \frac{1}{\sqrt{2}} \left[ \begin{matrix} 
            I_m  & i J_m  \\   J_m & i I_m
            \end{matrix} \right] \quad \hbox{and} \quad 
  L_{2m} :=  \frac{1}{\sqrt{2}}  \left[ \begin{matrix} 
          I_m & 0 & J_m \\  0 & \sqrt{2} & 0 \\  I_m & 0 & - i J_m 
            \end{matrix} \right]. 
$$

\begin{prop}
The matrix $L_n$ is unitary, that is, $L_n L_n^* = I_{n+1}$, and it satisfies
$$
   L_n L_n^\tr = L_n^\tr L_n = J_{n+1}.
$$ 
Furthermore, the polynomials $\PP_n^\tr= \{P_{k,n}: 0\le k \le n\}$ and $\QQ_n^\tr= \{Q_{k,n}: 0\le k \le n\}$ in the
Definition \ref{def:PvsQ} are related by 
\begin{equation} \label{eq:Q=LP}
     \QQ_n= L_n \PP_n \qquad \hbox{and} \qquad \PP_n = L_n^* \QQ_n.
\end{equation}
\end{prop}

\begin{proof}
All properties follow directly from straightforward  matrix multiplication. 
\end{proof}

\begin{thm}
Let $\{P_{k,n}: 0 \le k \le n\}$ and $\{Q_{k,n}: 0\le k \le n\}$ be polynomials given in Definition \ref{def:PvsQ}.
\begin{enumerate}[   (1)]
\item $\{Q_{k,n}: 0 \le k \le n\}$ is a basis of $\CV_n^2(W, \CC)$ that satisfy \eqref{eq:Qconjugate}
if and only if $\{P_{k,n}: 0 \le k \le n\}$ is a basis of $\CV_n^2(W)$.  
\item $\{Q_{k,n}: 0 \le k \le n\}$ is an orthonormal basis of $\CV_n^2(W, \CC)$ if and only if 
$\{P_{k,n}: 0 \le k \le n\}$ is an orthonormal basis of $\CV_n^2(W)$. 
\item The reproducing kernels of $\CV_n^2(W)$ and $\CV_n^2(W,\CC)$ agree; in particular,
\begin{equation}\label{K=KC}
    \Pb_n^\CC(z, \zeta) = \Pb_n( (x,y), (u,v)) \quad \hbox{and} \quad 
    \Kb_n^\CC(z, \zeta) = \Kb_n( (x,y), (u,v)), 
\end{equation}
where $z = x+iy, \quad \zeta = u+iv$.
\end{enumerate}
\end{thm}

\begin{proof}
If $f$ and $g$ are real valued polynomials, then $\la f, g\ra_W^\CC = \la f, g\ra_W$. Since, by definition, 
$P_{k,n}(x,y) = \Re \{Q_{k,n}(z,\bar z)\}$ for $0 \le k \le n/2$ and $P_{k,n}(x,y) = \Im \{Q_k^n(z,\bar z)\}$ 
for $n/2 < k \le n$, all $P_{k,n}$ are real valued polynomials in $\Pi_n^2$. On the other hand, given $\{P_{k,n}\}$, 
the definition of $\{Q_{k,n}\}$ shows that \eqref{eq:Qconjugate} is satisfied and $Q_{k,n} \in \Pi_n^2(\CC)$. 
Furthermore, since $\QQ_n = L_n \PP_n$, we have  
\begin{align*}
   \la \QQ_n, \QQ_m^\tr \ra_W^\CC =  L_n \la \PP_n \PP_m^\tr \ra_W L_m^*,
\end{align*}
which shows that $\{Q_{k,n}\}$ is a basis of $\CV_n^2(W, \CC)$ if and only if $\{P_{k,n}\}$ is a basis of $\CV_n^2(W)$. 
If $\{P_{k,n}\}$ is an orthonormal basis, then $\la \PP_n \PP_n^\tr \ra_W= I_{n+1}$, so that, by $L_nL_n^* = I_{n+1}$,
$\la QQ_n, \QQ_n^\tr \ra_W^\CC =I_{n+1}$ and $\{Q_{k,n}\}$ is orthonormal. Since $L^n$ is unitary, the relation
is reversible. This completes the proof of assertions (1) and (2). 

With $z = x+iy$ and $\zeta = u+iv$ and using $L_m^\tr  \overline{L_m} = (L_m^* L_m)^\tr = I_{m+1}$, we obtain
$$
  [\QQ_m (z,\bar z)]^\tr \overline {\QQ_m (w, \bar w) } = [\PP_m(x,y)]^\tr L_m^\tr \overline{L_m} \PP_m(u,v) = [\PP_m(x,y)]^\tr 
  \PP_m(u,v); 
$$
the left hand side is the reproducing kernel of $\CV_m^2(W,\CC)$ while the right hand side is the reproducing kernel
of $\CV_m^2(W)$. Summing over $m$ proves \eqref{K=KC}. 
\end{proof}

Since the integral measure is the same, the convergence of the Fourier orthogonal expansions in either complex variable or two real variables should be the same. The identity \eqref{K=KC} not
only confirms this conception, it also shows that the reproducing kernels are identitical. In particular, the kernels
$\Pb_n^\CC(z,\bar z)$ and $\Kb_n^\CC(z,\bar z)$ are real valued. 

Let us revisit our examples on the Hermite polynomials and disk polynomials. Notice that the polar coordinate 
$z = r e^{i \t}$ in $\CC$ is equivalent to the polar coordinates $(x,y) = (r\cos \t, r \sin \t)$ by the Euler formula of
$e^{i\t} = \cos \t + i \sin \t$. 

\medskip\noindent
{\bf Example 4.1.} {\it Hermite polynomials}. Let $H_{k,j}$ be the complex Hermite polynomials in Example 3.1. Then
$H_{n-j,j} \in \CV_n^2(W, \CC)$. By \eqref{HermiteC-3}, we see that
\begin{align*}
   \Re \{H_{n-j,j}(z,\bar z)\} & = (-1)^j j! L_j^{(n-j)}( r^2) r^{n-2j} \cos (n-2j)\theta, \quad 0 \le j \le \f n 2  \\
   \Im \{H_{n-j,j} (z,\bar z)\} & = (-1)^j j! L_j^{(n-j)}( r^2)  r^{n-2j} \sin  (n-2j)\theta,   \quad 0 \le j < \f n 2, 
\end{align*}
which is, up to a constant, exactly the orthogonal polynomials \eqref{Hermite2} of two real variables in Example 2.1.
This also verifies \eqref{eq:PvsQ} up to a normalization constant.  

\qed

\medskip\noindent
{\bf Example 4.2.} {\it Disk polynomials}. Let $P_{k,j}$ be the complex disk polynomials in Example 3.2. Then
$P_{n-j,j}^\l \in \CV_n^2(W, \CC)$. By \eqref{disk-poly2},
\begin{align*}
   \Re \{P_{n-j,j}^\l (z,\bar z)\} & = \frac{j}{(\l+1)_j}P_j^{(\l,n-2j)}( 2 r^2-1) r^{n-2j} \cos (n-2j)\theta, \quad 0 \le j \le \f n 2  \\
   \Im \{P_{n-j,j}^\l (z,\bar z)\} & = \frac{j}{(\l+1)_j}P_j^{(\l,n-2j)}( 2 r^2-1) r^{n-2j}  \sin  (n-2j)\theta,   \quad 0 \le j < \f n 2, 
\end{align*}
which is, up to a constant, exactly the orthogonal polynomials \eqref{Disk2} of two real variables in Example 2.2.
This also verifies \eqref{eq:PvsQ} up to a normalization constant. Using the result from real disk polynomials,
we have, for example, the following relation: 
\begin{align*}
  \sum_{k+j =n} & \frac{P_{k,j}^\l (z,\bar z) \overline{P_{k,j}^\l (\zeta,\bar \zeta)}}{h_{j,k}^\l}\\
    & = \frac{n+\mu+\f12}{\mu+\f12} c_\mu \int_{-1}^1 C_n^{\mu + \f12} \left(\Re\{z \bar \zeta\} + \sqrt{1-|z|^2} \sqrt{1-|\zeta|^2} t\right)
     (1-t^2)^{\mu -1} dt. 
\end{align*}
Indeed, the left hand side of this identity is the reproducing kernel of $\CV_n^d$, so that this is the
identity in Corollary 6.1.10 of \cite{DX} written in complex variables. 

\medskip\noindent
{\bf Example 4.3.}
{\it Chebyshev polynomials on the region bounded by the deltoid}. These polynomials are orthogonal with respect to
\begin{equation}
  w_\a(z): = \left [-3(x^2+y^2 + 1)^2 + 8 (x^3 - 3 xy^2)  +4\right]^{\a}, \quad \a = \pm \frac12,
\end{equation}
on the deltoid, which is a region bounded by the Steiner's hypocycloid $-3(x^2+y^2 + 1)^2 + 8 (x^3 - 3 xy^2)  + 4 =0$ that
can be described as the curve  
$$
  x + i y = (2 e^{ i \t} + e^{- 2 i \t})/3, \qquad  0 \le \t \le 2 \pi.
$$ 
The three-cusped region is depicted in Figure \ref{figure:region}. 
\begin{figure}
\begin{center} 
 \includegraphics[scale=0.48]{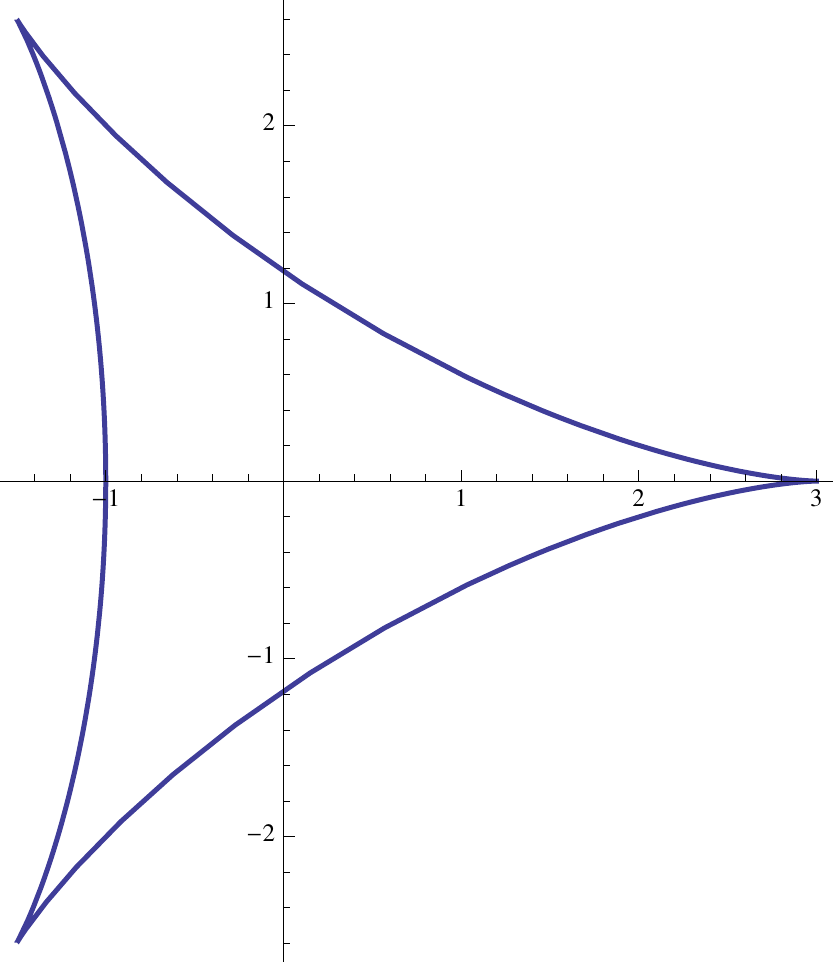} 
\label{figure:region} 
\end{center} 
\end{figure} 
These polynomials are first studied by Koornwinder in \cite{K} and they are related to the symmetric and
antisymmetric sums of exponentials on a regular hexagonal domain \cite{LSX}. In stead of stating their 
explicit formulas, it suffices to define these polynomials recursively. Let $T_k^n \in \CV_n^2(w_{-\frac12}, \CC)$
and $U_k^n \in  \CV_n^2(w_{\frac12}, \CC)$ be the Chebyshev polynomials of the first and the second kind,
respectively, defined by the recursive relations
\begin{align} \label{recurT}
     P _k^{n+1} (z,\bar z) = 3 z P_{k}^n(z,\bar z) -
                 P_{k+1}^n(z,\bar z) - P_{k-1}^{n-1}(z,\bar z) 
\end{align}
for $ 0 \le k \le n$ and $n\ge 1$, where $P_k^n$ is $T_k^n$ or $U_k^n$ as determined by 
\begin{align*}
&T_{-1}^n(z,\bar z ) : = T_1^{n+1}(z,\bar z ), \quad T_{n+1}^n(z,\bar z )
  : = T_{n}^{n+1}(z,\bar z ),\\
&U_{-1}^n(z,\bar z ) : =  0, \quad  U_n^{n-1}(z,\bar z ): =0,
\end{align*}
and, moreover, 
\begin{align*}
& T_0^0(z,\bar z)=1, \quad T_0^1(z,\bar z)= z,  \quad T_1^1(z,\bar z)= \bar z,\\
& U_0^0(z,\bar z)=1, \quad U_0^1(z,\bar z)= 3z,  \quad U_1^1(z,\bar z)= 3 \bar z.
\end{align*}
Then these polynomials satisfy the relation 
$$
  P_k^n (z, \bar z) = \overline{P_{n-k}^n (z, \bar z)}, \quad 0 \le k \le n. 
$$
Furthermore, $\{T_k^n(z, \bar z): 0 \le k \le n\}$ is a mutually orthogonal basis of  $\CV_n^2(w_{-\frac12}, \CC)$
and $\{U_k^n(z,\bar z): 0 \le k \le n\}$  is a mutually orthogonal basis of  $\CV_n^2(w_{-\frac12}, \CC)$.

This family of polynomials is known explicitly only in complex variables, although a real basis can be deduced from
\eqref{eq:PvsQ}.
\qed

\section{Structural relations of orthogonal polynomials} 
\setcounter{equation}{0}

Three--term relations for complex orthogonal polynomials are different from those for real orthogonal polynomials. 
In the following we normalize $W$ so that $\QQ_{0}(z,\bar z ) =1$ and we define $\QQ_{-1}(z,\bar z) =0$. Let
$\CM^\CC(n,m)$ denote the set of complex matrices of size $n \times m$. 

\begin{thm}
For $n \in \NN_0^d$, let $\QQ_n = \{Q_{k,n}:0 \le k \le n\}$ be a basis of $\CV_n^d(W, \CC)$ that satisfies 
\eqref{eq:Qconjugate2}. Then there are matrices $\a_n \in \CM^\CC(n+1,n+2)$, $\b_n \in \CM^\CC (n+1, n+1)$ 
and $\g_n \in \CM^\CC(n+1, n)$ such that 
\begin{align}\label{3-termC}
  z \QQ_n (z,\bar z) = \a_n \QQ_{n+1} (z,\bar z)+ \b_n \QQ_n (z,\bar z) + \g_{n-1} \QQ_{n-1} (z,\bar z),
\end{align}
where, setting $H_n = \la  \QQ_n, \QQ_n^\tr \ra_W$, then $\g_n$ satisfies
\begin{equation}\label{gamma-alpha}
    \g_{n-1} H_{n-1} = J_{n+1} (\a_{n-1} H_n)^\tr J_n.
\end{equation}
\end{thm}
 
\begin{proof}
Since $z \QQ_n$ is a polynomial of degree $n+1$, it can be written as a linear combination of $\QQ_{n+1}, \QQ_n, 
\ldots, \QQ_0$. The orthogonality then implies three--terms relations. Furthermore, we have 
\begin{align*}
   \la z \QQ_n, \QQ_{n+1}^\tr \ra_W =  \a_n H_{n+1}, \quad  \la z \QQ_n, \QQ_n^\tr \ra_W =  \b_n H_n,\quad
    \la z \QQ_n, \QQ_{n-1}^\tr \ra_W =  \g_n H_{n-1}.
\end{align*}
In particular,  by \eqref{eq:Qconjugate2}, we see that 
\begin{align*}
 \la z \QQ_n, \QQ_{n-1}^\tr \ra_W = & \int_\Omega z \QQ_n (z,\bar z)\overline{\QQ_{n-1}(z,\bar z)^\tr} w(z) dxdy \\
     = &  J_{n+1}  \int_\Omega z \overline {\QQ_n(z,\bar z)} {\QQ_{n-1}(z,\bar z)^\tr} w(z) dxdy J_n 
      = J_{n+1} \la z \QQ_{n-1} \QQ_n\ra_W^\tr J_n. 
\end{align*}
Since $H_n$ is invertible, this verifies \eqref{gamma-alpha}. 
\end{proof}

For a matrix $M \in \CM(n, m)$, we define a matrix $M^\vee$ by
$$
   M^\vee: = J_n \overline{M} J_m.
$$
Taking conjugate of \eqref{3-termC} and applying \eqref{eq:Qconjugate2}, we see that $\QQ_n$ also satisfies
\begin{align}\label{3-termC2}
  \bar z \QQ_n (z,\bar z) = \a_n^\vee \QQ_{n+1} (z,\bar z)+ \b_n^\vee \QQ_n (z,\bar z) + \g_{n-1}^\vee \QQ_{n-1} (z,\bar z).
\end{align}

In the case that $\QQ_n$ consists of an orthonormal basis of $\CV_n^d(W,\CC)$, the matrix $H_n$ is an identity and 
the relation between $\a_n$ and $\g_n$ can be written as
\begin{equation} \label{eq:a-g2}
  \g_{n-1}^\vee =   \a_{n-1}^*\quad \hbox{or} \quad \g_{n-1} = (\a_{n-1}^*)^\vee. 
\end{equation}

The three--term relation \eqref{3-termC} can also be derived from the three--term relations of real orthogonal 
polynomials. 

\begin{prop}
Let $\PP_n$ consists of orthonormal basis of $\CV_n^d(W)$. Assume that $\{\PP_n\}$ satisfies the three--term
relations \eqref{three-term}. If $\QQ_n$ and $\PP_n$ are related by \eqref{eq:Q=LP}, then the coefficients of 
the three--term relation \eqref{3-termC} can be expressed in the coefficients of  \eqref{three-term} as follows:
\begin{align}\label{3-term-coeff}
   \a_n = L_n (A_{n,1}+ i A_{n,2}) L_{n+1}^* \quad \hbox{and}\quad \b_n = L_n (B_{n,1}+ i B_{n,2}) L_n^*.
\end{align}
In particular, if $W$ is centrally symmetric, then $\b_n = 0$ for all $n$. 
\end{prop}
 
\begin{proof}
Setting $z = x + i y$ and $\QQ_n = L_n \PP_n$ in \eqref{3-termC}, we can expand $z \QQ_n$ by the three--term 
relations \eqref{three-term} for $\PP_n$ and using $\PP_n = L_n^* \QQ_n$ to obtain \eqref{3-termC}.  
\end{proof} 

The connection between the two three--term relations allow us to state  Favard's theorem for complex orthogonal 
polynomials. 

\begin{thm} \label{thm:FavardC}
Let $\{\QQ_n\}_{n=0}^\infty = \{Q_{k,n}: 0 \le k \le n, n\in \NN_0\}$, $\QQ_0 =1$, be an arbitrary sequence in $\Pi^2(\CC)$. 
Then the following statements are equivalent.
\item{}\hspace{.1in} (1).  There exists a positive definite linear functional $\CL$ on $\Pi^2(\CC)$ which 
makes $\{\QQ_n\}_{n=0}^\infty$ an orthonormal basis in $\Pi^d(\CC)$.
\item{}\hspace{.1in} (2). For $n \ge 0$, $1\le i\le d$, there exist matrices $\a_{n}: (n+1)\times (n+2)$ and
$b_{n}: (n+1)\times (n+1)$ such that 
\begin{align} \label{three-termC}
   z \QQ_n (z,\bar z) = \a_n \QQ_{n+1} (z,\bar z)+ \b_n \QQ_n (z,\bar z) + (\a_{n-1}^*)^\vee \QQ_{n-1} (z,\bar z),
\end{align}
and the matrices in the relation satisfy the rank condition 
$$
   \rank (\a_n + \a_n^\vee)  = \rank (\a_n - \a_n^\vee) = n+1 \quad \hbox{and} \quad \rank \left [ \begin{matrix} \a_n \\ \a_{n}^\vee \end{matrix} \right]=n+2.
$$
\end{thm}

\begin{proof}
From \eqref{3-term-coeff}, we immediately deduce that 
$$
  A_{n,1} = \frac 12 (L_n \a_n L_{n+1}^* + \overline{L_n \a_nL_{n+1}^*})\quad\hbox{and}\quad
    A_{n,2} = \frac 1{2i} (L_n \a_n L_{n+1}^* - \overline{L_n \a_nL_{n+1}^*}).
$$
Since $L_n L_n^\tr = J_n$, it follows that 
\begin{equation} \label{3-term-coeff2}
   L_n^* A_{n,1} L_{n+1} = \frac 12 ( \a_n + \a_n^\vee) \quad\hbox{and}\quad 
   L_n^* A_{n,2} L_{n+1} = \frac 1{2i}( \a_n - \a_n^\vee), 
\end{equation}
which also lead to
$$
  \left [\begin{matrix} L_n &\bigcirc \\ \bigcirc& L_n \end{matrix} \right]
   \left [ \begin{matrix} A_{n,1} \\ A_{n,2} \end{matrix} \right]
   =  \frac 12 \left[ \begin{matrix} I_n & I_n \\ i I_n & - i I_n \end{matrix} \right]
   \left [ \begin{matrix} \a_n \\ \a_{n}^\vee \end{matrix} \right].   
$$
These relations allow us to translate the rank conditions on the matrices $A_{n,i}$ in Theorem \ref{thm:nFavard} to
matrices $\a_n$ and $\a_n^\vee$. 
\end{proof}

The three--term relations for $\PP_n$ satisfy additional relations, called commuting conditions, which comes from
the fact that the associated block Jacobi matrices $\CJ_i$ commute, where 
\begin{equation*} 
\CJ_i = \left[ \begin{matrix} B_{0,i}&A_{0,i}&&\bigcirc\cr
A_{0,i}^\tr&B_{1,i}&A_{1,i}&&\cr
&A_{1,i}^\tr&B_{2,i}&\ddots\cr
\bigcirc&&\ddots&\ddots \end{matrix}
\right],\qquad i =1, 2.
\end{equation*} 
These commuting conditions translate to conditions on $\a_n$ and $\b_n$. Without getting into details, 
we record them below. 

\begin{prop}
For orthonormal polynomials $\QQ_n$, the coefficients of the three--term relation \eqref{3-termC} satisfy
\begin{align*}
    \a_n \a_{n+1}^\vee & = \a_n^\vee \a_{n+1}, \\
    \a_n \b_{n+1}^\vee + \b_n \a_n^\vee & = \b_n^\vee \a_{n} + \a_n^\vee \b_{n+1}, \\
    \a_n \a_n^* + \b_n \b_n^* + (\a_{n-1}^*)^\vee \a_{n-1}^\vee & = \a_{n-1}^\vee (\a_{n-1}^\vee)^*  + \b_n^* \b_n+ \a_{n}^* \a_{n}. 
\end{align*}
\end{prop}

Another result worth mentioning is the Christoffel-Darboux formula stated in the following: 

\begin{prop}
For orthonormal polynomials $\QQ_n$, we have
$$
   \Kb_n^\CC(z, \zeta) = \frac{\QQ_{n+1}(z,\bar z)^* \a_n^\tr J_{n+1} \QQ_n(\zeta, \bar \zeta) - 
      \QQ_{n+1}(\zeta,\bar \zeta)^* \a_n^\tr J_{n+1}  \QQ_n(z, \bar z)}{z- \zeta} .
$$
\end{prop}

\begin{proof}
Recall that $\Kb_n^\CC(x+iy,u+iv) = K_n((x,y), (u,v))$. By \eqref{CDformula} and \eqref{3-term-coeff}, 
\begin{align*}
  (z-\zeta) \Kb_n^\CC(z,\zeta) & = (x -u) K_n((x,y), (u,v))+ i (y-v)K_n((x,y), (u,v)) \\
   & = (L_n^* \a_n L_{n+1} \PP_{n+1}(x,y))^\tr \PP_n(u,v) -  (L_n^* \a_n L_{n+1} \PP_{n+1}(u,v))^\tr \PP_n(x,y) \\
   & =  (L_n^* \a_n \QQ_{n+1}(z,\bar z))^\tr L_n^* \QQ_n(\zeta,\bar \zeta) - 
           (L_n^* \a_n  \QQ_{n+1}(\zeta,\bar \zeta))^\tr L_n^* \QQ_n(z,\bar z),
\end{align*}
which simplifies, since $L_n L_n^\tr = J_{n+1}$, to the desired identity.
\end{proof}

Our last result in this section is about common zeros of $\QQ_n$. We call $z$ a common zeros of $\QQ_n$ if
every component of $\QQ$ vanishes at $z$, that is, $Q_{k,n}(z,\bar z ) = 0$ for $0 \le k \le n$. For $\PP_n$, it
is known that it has at most $\dim \Pi_{n-1}^2 = \binom{n+1}2$ common zeros and it has $\dim \Pi_{n-1}^2$
zeros if and only if $A_{n-1,1} A_{n-1,2}^\tr = A_{n-1,1}^\tr A_{n-1,2}^\tr$. We can convert these results to
complex orthogonal polynomials. 

\begin{thm} \label{thm:zeros}
Assume $\QQ_n$ consists of an orthonormal basis of $\CV_n^d(W,\CC)$. Then 
\begin{enumerate}[   1.]
\item $\QQ_n$ has at most $\dim \Pi_{n-1}^2$ common zeros. 
\item $\QQ_n$ has $\dim \Pi_{n-1}^2$ zeros if and only if 
\begin{equation}\label{max-zero-cond}
    \a_{n-1} \a_{n-1}^* J_{n+1} = ( \a_{n-1} \a_{n-1}^* J_{n+1} )^\tr. 
\end{equation}
\end{enumerate}
\end{thm}

\begin{proof}
From $\QQ_n = L_n \PP_n$ it follows that $z = x+iy$ is a zero of $\QQ_n$ if and only if $(x,y)$ is a zero of $\PP_n$,
so that the results follow from that of $\PP_n$. By \eqref{3-term-coeff2}, 
\begin{align*}
  4  i A_{n-1,1}A_{n-1,2}^\tr = &\, L_{n-1}^* (\a_{n-1} + \a_{n-1}^\vee)L_n L_n^\tr  (\a_{n-1} - \a_{n-1}^\vee)^\tr \overline{L_n} \\
    = &\, L_{n-1}^* (\a_{n-1} + \a_{n-1}^\vee)J_{n+1}  (\a_{n-1} - \a_{n-1}^\vee)^\tr \overline{L_n},
\end{align*}
from which it follows that $A_{n-1,1} A_{n-1,2}^\tr = A_{n-1,1}^\tr A_{n-1,2}^\tr$ is equivalent to 
$$
  (\a_{n-1} + \a_{n-1}^\vee)J_{n+1}  (\a_{n-1}^\tr - (\a_{n-1}^\vee)^\tr) 
     =   (\a_{n-1} - \a_{n-1}^\vee)J_{n+1}  (\a_{n-1}^\tr + (\a_{n-1}^\vee)^\tr ),
$$
which simplifies to \eqref{max-zero-cond}. 
\end{proof}

The existence of maximal number of common zeros of $\QQ_n$ implies the existence of a Gaussian cubature rule of
degree $2n-1$, which is important for numerical analysis and several other topics. 

\begin{prop}
Let $\QQ_n$ consist of an orthonormal basis of $\CV_n^d(W_\mu)$. Then $z \in \CC$ is a common zero of
$\QQ_n$ if $z$ is an eigenvalue of the matrix
$$
  \CJ_n := \left[ \begin{matrix} \b_0& \a_0 &&&\bigcirc\cr
       (\a_0^\vee)^* &\b_1 & \a_1 &&\cr  &\ddots&\ddots&\ddots&\cr
     &&  (\a_{n-3}^\vee)^* & \b_{n-2}& \a_{n-2}\cr
     \bigcirc&&&  (\a_{n-2}^\vee)^*  & \b_{n-1}\end{matrix} \right]
$$ 
with eigenvector $\xi_z: = (\QQ_0(z,\bar z)^\tr, \QQ_1(z,\bar z)^\tr, \ldots, \QQ_{n-1}(z,\bar z)^\tr)^\tr$. 
\end{prop}

\begin{proof}
If $z$ is a common zero of $\QQ_n(z,\bar z)$, then the three--term relation that involves $\QQ_n(z, \bar z)$ becomes
$$
   (\a_{n-2}^\vee)^* \QQ_{n-2}(z,\bar z) + \b_{n-1} \QQ_{n-1}(z,\bar z) = z \QQ_{n-1}(z,\bar z),
$$
which, together with \eqref{three-termC} for $k =0,1,\ldots, n-2$ shows that $\CJ_n \xi_z = z \xi_z$, so that $z$ is
an eigenvalue of $\CJ_n$. 
\end{proof}

One natural question is if the inverse of the above proposition holds;  that is, if every eigenvalue of $\CJ_n$ is a 
zero of $\QQ_n$. The answer is no and the reason is that if $z$ is an eigenvalue of $\CJ_n$, then $\bar z$ is also
an eigenvalue of $\CJ_n$ with an eigenvector $(\QQ_0(z,\bar z)^\tr, J_2 \QQ_1(z,\bar z)^\tr, \ldots, J_n \QQ_{n-1}(z,\bar z)^\tr)^\tr$,
as can be seen by \eqref{3-termC2}, \eqref{eq:a-g2} and \eqref{eq:Qconjugate2}. As a result, we see that if $\lambda$ 
is an eigenvalue of $\CJ_n$ with eigenvector  $\xi = (\xi_0, \xi_1^\tr, \ldots, \xi_{n-1}^\tr)^\tr$, where $\xi_j \in \CC^{j+1}$,
then $\lambda$ is a common zero of $\QQ_n$ only if $\overline{\xi_j} = J_{j+1} \xi_j$ for $j =1,2,\ldots, n-1$. 

\medskip\noindent
{\bf Example 5.1.} {\it Hermite polynomials}. Let $H_{k,j}$ be the complex Hermite polynomials in Example 3.1. The three-term
relation of these polynomials is given in \eqref{HermiteC-4}. Let $Q_{k,n} (z,\bar z) = H_{k,n-k}(z, \bar z)/\sqrt{k!(n-k)!}$. By
\eqref{HermiteC-5}, $\{Q_{k,n}: 0 \le k \le n\}$ is an orthonormal basis of $\CV_n^d(W_H)$ for which the three-term relation 
\eqref{three-termC} takes the form
$$ 
   z \QQ_n = \left[ \begin{matrix} 0 & 1 & & & \bigcirc\\ 0 & & \sqrt{2} && \\ \vdots &\bigcirc & & \ddots & \\
       0 & & & & \sqrt{n+1} \end{matrix} \right]\QQ_{n+1}  + 
   \left[ \begin{matrix} \sqrt{n} & &   \bigcirc\\  & \ddots  & \\ \bigcirc  & & 1 \\  0 & \ldots &  0 
         \end{matrix} \right]\QQ_{n-1}.  
$$ 
Since $H_n$ is an identity matrix, the relation \eqref{gamma-alpha} clearly holds.  \qed 

\medskip

\medskip\noindent
{\bf Example 5.2.} {\it Disk polynomials}. Let $H_{k,j}$ be the complex Hermite polynomials defined in Example 3.2. 
The three-term relation of these polynomials is given in \eqref{disk-poly3}. Let $Q_{k,n}(z,\bar z)  = P_{k,n-k}^\l(z, \bar z)/
\sqrt{h_{k,n-k}^\lambda}$.  By \eqref{disk-poly4}, $\{Q_{k,n}: 0 \le k \le n\}$ is an orthonormal basis of $\CV_n^d(W_\l)$
for which \eqref{disk-poly3} can be rewritten as
$$
   z Q_{k,n}(z,\bar z) = a_k^n Q_{k+1,n}(z,\bar z) + a_{n-k-1}^{n-1} Q_{k, n-1}(z,\bar z), 
$$
where 
$$
       a_k^n: = \sqrt{\frac{(\l+k+1)(k+1)}{(\l+n+1)(\l+n+2)}}, \qquad 0 \le k \le n. 
$$
Putting In matrix form, the relation takes the form
$$ 
   z \QQ_n = \left[ \begin{matrix} 0 & a_0^n  & & \bigcirc\\ \vdots &\bigcirc &  \ddots & \\
       0 &  & & a_n^n \end{matrix} \right] \QQ_{n+1}  + 
   \left[ \begin{matrix} a_{n-1}^{n-1} & &   \bigcirc\\  & \ddots  & \\ \bigcirc  & & a_0^{n-1} \\  0 & \ldots &  0 
         \end{matrix} \right]\QQ_{n-1}.  
$$ 
which is the three-term relation \eqref{three-termC}.  \qed

\medskip

In both of the above examples, the matrix $\b_n = 0$ since the weight functions are centrally symmetric. Notice that 
the condition \eqref{max-zero-cond} is not satisfied in both cases, so that the polynomials in $\QQ_n$ do not have 
maximal common zeros. In fact, in the centrally symmetric case, it is known that polynomials in $\PP_n$, since 
those in $\QQ_n$, do not have any common zero if $n$ is even and have a single common zero if $n$ is odd. 

\medskip\noindent
{\bf Example 5.3.} {\it Chebyshev polynomials on the region bounded by the deltoid.}  Both families, $T_k^n(z,\bar z)$ and $U_k^n(z,\bar z)$, satisfy
the three-term relations given by \eqref{recurT}. Each family is mutually orthogonal and the normalization constants of these polynomials 
are given in  (5.6) and (5.7) of \cite{LSX}. Let $\wt T_k^n(z,\bar z)$ and $\wt U_k^n(z,\bar z)$ denote the
orthonormal polynomials. Then the three-term relation \eqref{three-termC} becomes
$$
   3 z \TT_n = \left[ \begin{matrix} 1 & & &  \bigcirc & 0 \\    &  \ddots & & & \vdots \\
      &  & 1 &  & 0 \\
    \\ \bigcirc  &  & &   \sqrt{2} & 0\end{matrix} \right] \TT_{n+1}  + \b_n
     \TT_n
   + \left[ \begin{matrix} 0 & 0 &\cdots & 0 \\
           \sqrt{2} & & &  \bigcirc\\ & 1 && \\ & & \ddots  & \\ \bigcirc  & & & 1
         \end{matrix} \right]\TT_{n-1},
$$
where $\b_n =   \diag\{\sqrt{2}, 1, \ldots, 1,\sqrt{2} \}$ is a diagonal matrix, and 
$$
  3 z \UU_n =  \left[ \begin{matrix} I_n&0 \end{matrix}\right] \UU_{n+1} +  \left[ \begin{matrix} 0  & I_n \\
     0 & 0 \end{matrix} \right] \UU_n +  \left[ \begin{matrix} 0 \\ I_n \end{matrix} \right] \UU_{n-1}.
$$
It follows that the condition \eqref{max-zero-cond} is satisfied for $\UU_n$, which shows that polynomials
in $\UU_n$ have maximal number of common zeros by Theorem \ref{thm:zeros}. This was first established 
in \cite{LSX} using the explicit formulas for $U_k^n$. The condition \eqref{max-zero-cond}, however,
is not satisfied for $\TT_n$, which shows that $\TT_n$ does not have maximal number of common zeros. This
gives the first proof of this fact, which was verified in \cite{LSX}, using the explicit formulas of $T_k^n$, only 
for small $n$.
\qed




\begin{thebibliography}{BHW}

\bibitem{CGG}
         N. Cotfas, J. P. Gazeau, and K. G\'orska, 
         Complex and real Hermite polynomials and related quantizations, 
         \textit{J. Phys. A} \textbf{43} (2010), 305304 (14 pp).
	
\bibitem{DX}
         C. F. Dunkl and Y. Xu,
         \textit{Orthogonal polynomials of several variables},
         Cambridge Univ. Press, 2001.

\bibitem{Gh1}
         A. Ghanmi, 
         A class of generalized complex Hermite polynomials,
         \textit{J. Math. Anal. Appl.}, \textbf{340} (2008), 1395--1406.
        
\bibitem{Gh}
         A. Ghanmi, 
         Operational formulae for the complex Hermite polynomials $H_{p, q} (z, \bar{z})$,
         {\it Integral Transf. Special Func.}, 2013. DOI:10.1080/10652469.2013.772172
 
\bibitem{Inti}
         A. Intissar and A. Intissar,  Spectral properties of the Cauchy transform on $L^2(\CC;e^{|z|^2} d\l)$, 
         \textit{J. Math. Anal.  Appl.} \textbf{313} (2006), 400--418. 

\bibitem{Ismail}
         M. Ismail, 
         Analytic properties of complex Hermite polynomials, preprint, 2013. 
         
\bibitem{Ismail-S}         
         M. Ismail and P. Simeonov,
         Complex Hermite polynomials: their combinatorics and integral operators. 
         preprint, 2013. 
         
\bibitem{Ito}
         K. It\^o,
         Complex multiple Wiener integral,    \textit{Japan J. Math.} \textbf{22} (1952), 63--86. 
                 
\bibitem{K}
        T. Koornwinder, 
        Orthogonal polynomials in two varaibles which are eigenfunctions 
        of two algebraically independent partial differential operators,
        \textit{Nederl. Acad. Wetensch. Proc. Ser. A77} = \textit{Indag. Math}.
        \textbf{36} (1974), 357--381.      

\bibitem{LSX}
        H. Li, J. Sun and Y. Xu,
        Discrete Fourier analysis, cubature and interpolation on a hexagon and a triangle, 
        \textit{SIAM J. Numer. Anal.}  \textbf{46}, (2008), 1653--1681.
         
\bibitem{Suetin}
        P. K. Suetin, 
	\textit{Orthogonal polynomials in two variables},
        translated from the 1988 Russian original by E. V. Pankratiev, 
        Gordon and Breach, Amsterdam, 1999.

\bibitem{Szego}
	G. Szeg\H{o}, 
	\textit{Orthogonal polynomials}, 4th ed., 
        American Mathematical Society Colloquium Publication \textbf{23}, 
        American Mathematical Society, Providence, RI, 1975. 

\bibitem{Thanga}
        S. Thangavelu, 
	\textit{Lectures on Hermite and Laguerre expansions},
	Princeton University Press, Princeton, NJ, 1993.

\bibitem{TSH}
         K. Thirulogasanthar, N. Saad, and G. Honnouvo,
         2D-Zernike polynomials and coherent state quantization of the unit disc,
         arXiv:1303.5483, 2013.

\bibitem{W}
        A. W\"unsche, 
        Generalized Zernike or disc polynomials, 
        \textit{J. Comp. and Appl. Math.} \textbf{174} (2005) 135--163. 

\bibitem{Z}
        F. Zernike,         
       Beugungstheorie des schneidenver-fahrens und seiner verbesserten form, der phasenkontrastmethode. 
       \textit{Physica}, \textbf{1} (1934),  no. 7--12, 689--704.

\bibitem{Z2}
        F. Zernike and H. C. Brinkman, 
        Hypersph\"arishe Funktionen und die in sph\"arischen Bereichen orthogonalen Polynome, 
        \textit{Proc. Kon. Akad. v. Wet., Amsterdam} \textbf{38} (1935), 161--170. 

 \end{thebibliography}
\end{document}